\RequirePackage{amsmath}
\documentclass[runningheads]{llncs}
\usepackage{graphicx}
\graphicspath{{./figure/}}

\usepackage{amsfonts}
\usepackage{amsmath}
\usepackage{amssymb}
\usepackage{xcolor}

\begin{document}
\title{Parametric Fokker-Planck equation}
%
%
\author{Wuchen Li\inst{1} \and
Shu Liu\inst{2} \and Hongyuan Zha \inst{2} \and Haomin Zhou\inst{2}}
\authorrunning{Li, Liu, Zha, Zhou}
%
\institute{University of California, Los Angeles \and
Georgia Institute of Technology}
\maketitle              
\begin{abstract}
We derive the Fokker-Planck equation on the parametric space. It is the Wasserstein gradient flow of relative entropy on the statistical manifold. 
We pull back the PDE to a finite dimensional ODE on parameter space. 
Some analytical example and numerical examples  
are presented. 
\end{abstract}
\section{Introduction}
Fokker-Planck equation, a linear evolution partial differential equation (PDE), plays a crucial role in stochastic calculus, statistical physics and modeling \cite{Nelson2,QiMajda2017_lowdimensional,Risken1989_fokkerplanck}. Recently, people also discover its importance in statistics and machine learning \cite{LiuWang2016_stein,PavonTabakTrigila2018_datadriven,RezendeMohamed2015_variational}. Fokker-Planck equation describes the evolution of density functions of the stochastic process driven by a stochastic differential equation (SDE).

There is another viewpoint of Fokker-Planck equation based on optimal transport theory. It treats the equation as the gradient flow of relative entropy on probability manifold equipped with Wasserstein metric \cite{JKO,otto2001}. Recently, the studies have been extended to information geometry \cite{NG,IG,IG2}, creating a new area known as Wasserstein information geometry \cite{Li2018_geometrya,LiM,LiMontufar2018_ricci}. Inspired by those studies, in this paper, we derive the metric tensor on parameter space by pulling back the Wasserstein metric via the parameterized pushforward map. Then we compute the Wasserstein gradient flow (an ODE system) of relative entropy defined on parameter space. This leads to a statistical manifold version of Fokker Planck equation, which can be viewed as an approximation of the original PDE.

Our work is motivated by two purposes, (1) reducing the evolution PDE to a finite dimensional ODE system on parameter space; (2) applying parameterized pushforward map to obtain an efficient sampling method to generate samples from SDE. This is different from Markov Chain Monte Carlo (MCMC) methods \cite{LiuWang2018_stein} or momentum methods \cite{QiMajda2017_lowdimensional}. In this brief presentation, we sketch the theoretical framework with illustrations on several examples. The complete results will be reported in an extended version \cite{liu2020neural}. 

\section{Parametric Fokker-Planck equation}
In this section, we briefly review the fact that Fokker-Planck equation is a Wasserstein gradient flow of relative entropy. We then introduce a Wasserstein statistical manifold generated by parameterized mapping function. Based on it, we derive the parametric Fokker-Planck equation as the gradient flow of parameterized relative entropy. 
\subsection{Fokker-Planck equation}
Consider the Fokker-Planck equation:
\begin{equation}
    \frac{\partial\rho(t,x)}{\partial t}=\nabla\cdot(\rho(t,x)\nabla V(x))+\beta\Delta\rho(t,x), \quad \rho(0,x) = \rho_0(x). \label{FPE}
\end{equation}
Here $\nabla\cdot$, $\nabla$ is the divergence and gradient operator in $\mathbb{R}^d$, $\nabla V$ is the drift function and $\beta>0$ is a diffusion constant. There are several understandings for the equation \eqref{FPE}.

On the one hand, consider the stochastic differential equation:
\begin{equation}
    d\boldsymbol{X}_t = -\nabla V(\boldsymbol{X}_t)+\sqrt{2\beta}d\boldsymbol{B}_t,\quad X_0\sim\rho_0.
\end{equation}
Here $\{\boldsymbol{B}_t\}_{t\geq 0}$ is the standard Brownian motion. 
It is well known that the density function $\rho(t,x)$ of stochastic process $\boldsymbol{X}_t$, i.e. $\boldsymbol{X}_t\sim \rho(t,x)$, satisfies the Fokker-Planck equation \eqref{FPE}.

On the other hand, equation \eqref{FPE} is the Wasserstein gradient flow of relative entropy. Denote the probability space supported on $\mathbb{R}^d$:
\begin{equation*}
    \mathcal{P}=\left\{\rho\colon\int \rho(x)dx=1,~\rho(x)\geq 0,~\int |x|^2\rho(x)~ dx<\infty\right\}
\end{equation*}
Equipped with the Wasserstein metric \cite{Lafferty,otto2001}, $\mathcal{P}$ is an infinite dimensional Riemmanian manifold. Denote
$$T_\rho\mathcal{P}=\Big\{\dot\rho \colon \int \dot\rho(x)dx=0\Big\}.$$ Consider a specific $\rho\in\mathcal{P}$ and $\dot\rho_i\in T_\rho\mathcal{P}$, $i=1,2$. The Wasserstein metric tensor $g^W$ is defined as:
\begin{equation}
g^W(\rho)(\dot{\rho}_1,\dot{\rho}_2)=\int  \nabla\psi_1(x)\cdot\nabla\psi_2(x)\rho(x) ~dx,\nonumber\label{def_metric}
\end{equation}
where $\dot{\rho_i}=-\nabla\cdot(\rho_i\nabla\psi_i)$ for $i=1,2$.
Here $g^W$ is a metric tensor, which is a positive definite bilinear form defined on tangent bundle $T\mathcal{P}=\{(\rho,\dot\rho)\colon \rho\in \mathcal{P},~\dot\rho\in T_\rho\mathcal{P}\}$. 

The Riemannian gradient in $(\mathcal{P},g^W)$ is given as follows. Consider a smooth functional $\mathcal{F}\colon \mathcal{P}\rightarrow \mathbb{R}$, then
\begin{equation}
\begin{split}
\textrm{grad}_W\mathcal{F}(\rho)=&{g^{W}(\rho)}^{-1}\left(\frac{\delta\mathcal{F}}{\delta\rho}\right)(x)\\
  =&-\nabla\cdot(\rho(x) \nabla \frac{\delta}{\delta\rho(x)}\mathcal{F}(\rho)),
   \end{split}
   \label{gradflow}
\end{equation}
where $\frac{\delta}{\delta\rho(x)}$ is the $L^2$ first variation at variable $x\in\mathbb{R}^d$. In particular, consider the relative entropy 
\begin{equation}
\mathcal{F}(\rho)=\beta\int  \rho(x)\log\frac{\rho(x)}{e^{-\frac{V(x)}{\beta}}}~dx=\int V(x)\rho(x)dx+\beta\int\rho(x)\log\rho(x)dx.\label{relative entropy}
\end{equation}
Then $\nabla\left(\frac{\delta\mathcal{F}}{\delta\rho}\right)=\nabla V+\beta\nabla\log\rho$, and \eqref{gradflow} forms $$\frac{\partial \rho}{\partial t}=-\textrm{grad}_W\mathcal{F}(\rho)=\nabla\cdot(\rho\nabla V)+\beta\nabla\cdot(\rho\nabla\log\rho)).$$
Notice $\nabla\log\rho=\frac{\nabla\rho}{\rho}$, then $\nabla\cdot(\rho\nabla\log\rho)=\nabla\cdot(\nabla\rho)=\Delta\rho$. The above equation is exactly Fokker-Planck equation \eqref{FPE}. 

From now on, we apply the above geometric gradient flow formulation and derive the Fokker-Planck equation \eqref{FPE} on parameter space. 
\subsection{Parameter space equipped with Wasserstein metric}\label{2.3}
We consider a parameter space  $\Theta$ as an open set in $\mathbb{R}^m$. Denote the sample space $M=\mathbb{R}^d$. Suppose $T_\theta$ is a pushforward map from $M$ to $M$, which is parametrized by $\theta$. For example, we can set $T_\theta(x)=Ux+b$, with $\theta=(U,b),U\in GL_d(\mathbb{R}),~b\in\mathbb{R}^d$; we can also let $T_\theta$ be a neural network with parameter $\theta$. We further assume that $T_\theta$ is invertible and smooth with respect to parameter $\theta$ and variable $x$. 

Denote $p\in \mathcal{P}$ as a reference probability measure with positive density defined on $M$. For example, we can choose $p$ as the standard Gaussian. We denote $\rho_\theta$ as the density of ${T_\theta}_{\#}p$.\footnote{Let $X,Y$ be two measurable spaces, $\lambda$ is a probability measure defined on $X$; let $T:X\rightarrow Y$ be a measurable map, then $T_{\#}\lambda$ is defined as: $T_{\#}\lambda(E)=\lambda(T^{-1}(E))$ for all measurable $E\subset Y$. We call $T_{\#}p$ the pushforward of measure $p$ by map $T$.} We further require: $\int  |T_\theta(x)|^2~dp(x)<\infty$ holds for all $\theta\in\Theta$. Then $\rho_\theta\in\mathcal{P}$ for each $\theta\in\Theta$. Denote $\mathcal{P}_{\Theta}=\{\rho_\theta=\rho(\theta,x)|\theta\in\Theta\}$, then $\mathcal{P}_{\Theta}\subset \mathcal{P}$.

Now the connection between $\mathcal{P}$ and $\Theta$ is the pushforward operation $T_{\#}:\Theta\rightarrow\mathcal{P}_\Theta\subset\mathcal{P},\theta\mapsto\rho_\theta$. In order to introduce the Wasserstein metric to parameter space $\Theta$, we assume that $T_{\#}$ is an isometric immersion from $\Theta$ to $\mathcal{P}$. Under this assumption, the pullback $(T_{\#})^* g^W$ of the Wasserstein metric $g^W$ by $T_{\#}$  is the metric tensor on $\Theta$. Let us denote $G=(T_{\#})^*g^W $. Then for each $\theta$, $G(\theta)$ is a bilinear form on $T_\theta\Theta\simeq\mathbb{R}^m$, thus $G(\theta)$ can be treated as an $m\times m$ matrix. Computation of $G(\theta)$ is illustrated in the following theorem:
\begin{theorem}\label{thm_about_computing_metric_of_Theta}
Suppose $T_{\#}:\Theta\rightarrow \mathcal{P}$ is isometric immersion from $\Theta$ to $\mathcal{P}$. Then the metric tensor $G(\theta)$ at $\theta\in\Theta$ is $m\times m$ non-negative definite symmetric matrix and can be computed as:
\begin{equation}
  G(\theta) = \int  \nabla \boldsymbol{\Psi}(T_\theta(x))\nabla \boldsymbol{\Psi}(T_\theta(x))^T~dp(x),\label{Metric tensor D dimension}
\end{equation}
Or in entry-wised form:
\begin{equation}
  G_{ij}(\theta) =\int  \nabla\psi_i(T_\theta(x))\cdot\nabla\psi_j(T_\theta(x))~dp(x),~~1\leq i,j\leq m.\nonumber
\end{equation}
Here $\boldsymbol{\Psi}=(\psi_1,...\psi_m)^T$ and $\nabla\boldsymbol{\Psi}$ is $m\times d$ Jacobian matrix of $\boldsymbol{\Psi}$. For each $k=1,2,...,m$, $\psi_k$ solves the following equation:
\begin{equation}
  \nabla\cdot(\rho_\theta\nabla\psi_k(x)) = \nabla\cdot(\rho_\theta~\partial_{\theta_k}T_\theta(T^{-1}_\theta(x))).\label{Hodge Dcom}
\end{equation}
\end{theorem}
\begin{proof}
Suppose $\xi\in T\Theta$ is a vector field on $\Theta$, for a fixed $\theta\in\Theta$,
we first compute the pushforward $(T_{\#}|_\theta)_*\xi(\theta)$ of $\xi$ at point $\theta$: We choose any differentiable curve $\{\theta_t\}_{t\geq 0}$ on $\Theta$ with $\theta_0=\theta$ and $\dot{\theta}_0 = \xi(\theta)$. If we denote $\rho_{\theta_t}={T_{\theta_t}}_{\#}p$, then we have $(T_{\#})_*\xi(\theta) = \frac{\partial\rho_{\theta_t}}{\partial t}\Bigr|_{t=0}$. To compute $\frac{\partial\rho_{\theta_t}}{\partial t}\Bigr|_{t=0}$, we consider for any $\phi\in C^{\infty}_0(M)$:
\begin{align}
    \int  \phi(y)\frac{\partial \rho_{\theta_t}}{\partial t}(y)dy & = \frac{\partial}{\partial t}\left(\int  \phi(T_{\theta_t}(x))dp\right) = \int  \dot{\theta_t}^T \partial_\theta T_{\theta_t}(x)\nabla\phi(T_{\theta_t}(x))dp\nonumber \\ 
    & = \int \dot{\theta}_t^T \partial_\theta T_{\theta_t}(T^{-1}_{\theta_t}(x))\nabla\phi(x)~\rho_{\theta_t}(x)~dx\nonumber\\
    & = \int  \phi(x)\left(-\nabla\cdot(\rho_{\theta_t}\partial_\theta T_{\theta_t}(T_{\theta_t}^{-1}(x))^T~\dot{\theta}_t)\right)~dx\nonumber
\end{align}
This weak formulation reveals that
\begin{equation}
    (T_{\#}|_\theta)_*\xi(\theta)=\frac{\partial\rho_{\theta_t}}{\partial t}\Bigr|_{t=0}=-\nabla\cdot(\rho_\theta~\partial_\theta T_\theta(T_\theta^{-1}(x))^T~\xi(\theta))\label{compute_drhodt}
\end{equation}
Now let us compute the metric tensor $G$. Since $T_{\#}$ is isometric immersion from $\Theta$ to $\mathcal{P}$, the pullback of $g^W$ by $T_{\#}$ gives $G$, i.e. $(T_{\#})^*g^W = G$. By definition of pullback map, for any $\xi\in T\Theta$ and for any $\theta\in\Theta$, we have:
\begin{equation}
  G(\theta)(\xi(\theta),\xi(\theta)) = g^W(\rho_\theta)((T_{\#}|_\theta)_*\xi(\theta),(T_{\#}|_\theta)_*\xi(\theta)) \label{a}
\end{equation}
To compute the right hand side of (\ref{a}), recall \eqref{def_metric}, we need to solve for $\varphi$ from:
\begin{equation}
\frac{\partial\rho_{\theta_t}}{\partial t}\Bigr|_{t=0}=-\nabla\cdot(\rho_\theta\nabla\varphi(x)) \label{b}
\end{equation}
By \eqref{compute_drhodt}, \eqref{b} is:
\begin{equation}
\nabla\cdot(\rho_\theta\nabla\varphi(x)) = \nabla\cdot(\rho_\theta \partial_\theta T_\theta(T_\theta^{-1}(\cdot))^T~\xi(\theta))\label{equ}
\end{equation}
We can straightforwardly check that $\varphi(x) = \boldsymbol{\Psi}^T(x)\xi(\theta)$ is the solution of (\ref{equ}). Then $G(\theta)$ is computed as:
\begin{align}
  G(\theta)(\xi,\xi) &=\int  |\nabla\varphi(y)|^2~\rho_\theta(y)~dy =\int  |\nabla\varphi(T_\theta(x))|^2~dp(x)\nonumber\\
   &=\int |\nabla\boldsymbol{\Psi}(T_\theta(x))^T\xi|^2dp(x)=\xi^T\left(\int\nabla\boldsymbol{\Psi}(T_\theta(x))\nabla\boldsymbol{\Psi}(T_\theta(x))^T dp(x)\right)\xi\nonumber
\end{align}
Thus we can verify that: 
\begin{equation}
G(\theta) = \int  \nabla\boldsymbol{\Psi}(T_\theta(x))\nabla\boldsymbol{\Psi}(T_\theta(x))^T~dp(x) \nonumber
\end{equation}
\end{proof}
Generally speaking, the metric tensor $G$ doesn't have an explicit form when $d\geq 2$; but for $d=1$, $G$ has an explicit form and can be computed directly.
\begin{corollary}
When dimension $d$ of $M$ equals 1. And we further assume that:
$\rho_\theta>0$ on $M$ and $\lim_{x\rightarrow\pm\infty}\rho_\theta(x)=0$. Then $G(\theta)$ has an explicit form:
\begin{equation}
 G(\theta)=\int  \partial_\theta T_\theta(x)^T \partial_\theta T_\theta(x)~dp(x).\label{Metric tensor 1 dimension}
\end{equation}
\end{corollary}
The following theorem ensures the positive definiteness of the metric tensor $G$:
\begin{theorem}
We follow the notations and conditions in section 2.2,2.3. Then $G$ is Riemmanian metric on $T\Theta$ iff 
For each $\theta\in\Theta$, for any $\xi\in T_\theta\Theta~(\xi\neq 0)$, we can find $x\in M$ such that $\nabla\cdot(\rho_\theta~\partial_\theta T_\theta(T_\theta^{-1} (x)\xi)\neq 0$.
\end{theorem}
From now on, following \cite{LiM,LiMontufar2018_ricci}, we call $(\Theta, G)$ Wasserstein statistical manifold.

\subsection{Fokker-Planck equation on statistical manifold}
Recall the relative entropy functional $\mathcal{F}$ defined in (\ref{relative entropy}), we consider $F=\mathcal{F}\circ T_{\#}:\Theta\rightarrow \mathbb{R}$. Then:
\begin{equation}
F(\theta)=\mathcal{F}(\rho_\theta)=\int  V(x)\rho_\theta(x)~dx+\beta\int \rho_\theta(x)\log\rho_\theta(x)~dx.\label{relative entropy parameter}
\end{equation}
As in \cite{NG}, the gradient flow of $F$ on Wasserstein statistical manifold $(\Theta,G)$ satisfies 
\begin{equation}
\dot{\theta}=-G(\theta)^{-1}\nabla_\theta F(\theta).   \label{wass_grad_flow_on_para_spc}
\end{equation}
We call \eqref{wass_grad_flow_on_para_spc} {\em parametric Fokker-Planck equation}. 
The ODE (\ref{wass_grad_flow_on_para_spc}) as the Wasserstein gradient flow on parameter space $(\Theta,G)$ is closely related to Fokker-Planck equation on probability submanifold $\mathcal{P}_{\Theta}$. We have the following theorem, which is a natural result derived from submanifold geometry:
\begin{theorem}
Suppose $\{\theta_t\}_{t\geq 0 }$ solves (\ref{wass_grad_flow_on_para_spc}). Then $\{\rho_{\theta_t}\}$ is the gradient flow of $\mathcal{F}$ on probability submanifold $\mathcal{P}_{\Theta}$.
\end{theorem}

\section{Example on Fokker-Planck equations with quadratic potential}
The solution of Fokker-Planck equation on statistical manifold (\ref{wass_grad_flow_on_para_spc}) can serve as an approximation to the solution of the original equation (\ref{FPE}). However, in some special cases, $\rho_{\theta_t}$ exactly solves (\ref{FPE}). In this section, we demonstrate such examples. 

Let us consider Fokker-Planck equations with quadratic potentials whose initial conditions are Gaussian, i.e. 
\begin{equation}
V(x)=\frac{1}{2}(x-\mu)^T\Sigma^{-1}(x-\mu) \quad\mathrm{and}\quad\rho_0\sim\mathcal{N}(\mu_0,\Sigma_0).\label{conditions_example}
\end{equation}
Consider parameter space $\Theta=(\Gamma,b)\subset\mathbb{R}^{m}$ ($m=d(d+1)$), where $\Gamma$ is a $d\times d$ invertible matrix with $\det(\Gamma)>0$ and $b\in\mathbb{R}^d$. We define the parametric map as $T_\theta(x)=\Gamma x+b$. We choose the reference measure $p=\mathcal{N}(0,I)$.
Here is the lemma we have to use:
\begin{lemma}\label{coro_b}
Let $\mathcal{F}$ be the relative entropy defined in (\ref{relative entropy}) and $F$ defined in (\ref{relative entropy parameter}). For $\theta\in\Theta$, If the vector function $\nabla\left(\frac{\delta\mathcal{F}}{\delta\rho}\right)\circ T_\theta$ can be written as the linear combination of $\{ \frac{\partial T_\theta}{\partial \theta_1},...,\frac{\partial T_\theta}{\partial \theta_{m}}\}$, i.e. there exists $\zeta\in\mathbb{R}^m$, such that $\nabla\left(\frac{\delta\mathcal{F}}{\delta\rho}\right)\circ T_\theta(x)=\partial_\theta T_\theta(x)\zeta$. Then:\\ 
1) $\zeta=G(\theta)^{-1}\nabla_\theta F(\theta)$, which is the Wasserstein gradient of $F$ at $\theta$.\\
2) If we denote the gradient of $\mathcal{F}$ on $\mathcal{P}$ as $\mathrm{grad}\mathcal{F}(\rho_\theta)$ and the gradient of $\mathcal{F}$ on the submanifold  $\mathcal{P}_{\Theta}$ as $\mathrm{grad}\mathcal{F}(\rho_\theta)|_{\mathcal{P}_{\Theta}}$, then $ \mathrm{grad}\mathcal{F}(\rho_\theta)|_{\mathcal{P}_{\Theta}} = \mathrm{grad} \mathcal{F}(\rho_\theta)$.
\end{lemma}
\begin{proof}
The detailed proof is provided in \cite{Liu2019}. Here is an intuitive explanation: $\nabla\left(\frac{\delta\mathcal{F}}{\delta\rho}\right)=\nabla V+\beta\nabla\log\rho_{\theta}$ is the real vector field that moves the particles in Fokker-Planck equation; and $\partial_\theta T_\theta(T_\theta^{-1}(\cdot))\dot{\theta}$ is the approximate vector field induced by the pushforward map $T_\theta$. If such approximate is perfect with zero error, i.e. exits $\zeta$ such that $\nabla\left(\frac{\delta\mathcal{F}}{\delta\rho}\right)\circ T_\theta(x)=\partial_\theta T_\theta(x)\zeta$, then $\zeta=\dot{\theta}=G(\theta)^{-1}\nabla_\theta F(\theta)$ and the submanifold gradient agrees with entire manifold gradient.
\end{proof}
Now, let us come back to our example, we can compute
\begin{equation}
\rho_\theta(x)={T_\theta}_{\#}p(x)=\frac{f(T_\theta^{-1}(x))}{|\det(\Gamma)|}=\frac{f(\Gamma^{-1}(x-b))}{|\det(\Gamma)|},~ f(x)=\frac{\exp(-\frac{1}{2}|x|^2)}{(2p)^{\frac{d}{2}}}.\nonumber
\end{equation}
Then we have:
\begin{equation}
\nabla\left(\frac{\delta\mathcal{F}(\rho_\theta)}{\delta\rho}\right)\circ T_\theta(x)=\nabla (V+\beta\log\rho_\theta)\circ T_\theta(x) = \Sigma^{-1}(\Gamma x+b-\mu)-\beta\Gamma^{-T}x\nonumber
\end{equation}
is affine w.r.t. $x$.

Notice that $\partial_{\Gamma_{ij}}T_\theta(x)=(..0..\underset{i-\mathrm{th}}{x_j}..0..)^T$ and $\partial_{b_i}T_\theta=(..0..\underset{i-\mathrm{th}}{1}..0..)^T$. We can verify that $\zeta=(\Sigma^{-1}\Gamma-\beta\Gamma^{-T},\Sigma^{-1}(b-\mu))$ solves $\nabla\left(\frac{\delta\mathcal{F}(\rho_\theta)}{\delta\rho}\right)\circ T_\theta(x)=\partial_\theta T_\theta(x)\zeta$. By 1) of Corollary \ref{coro_b}, $\zeta = G(\theta)^{-1}\nabla_\theta F(\theta)$. Thus ODE (\ref{wass_grad_flow_on_para_spc}) for our example is:
\begin{align}
\dot{\Gamma}&=-\Sigma^{-1}\Gamma+\beta\Gamma^{-T}\quad \Gamma_0=\sqrt{\Sigma_0}\label{wass_grad_flow_1_example}\\
\dot{b}&=\Sigma^{-1}(\mu-b)\quad b_0=\mu_0\label{wass_grad_flow_2_example}
\end{align}
By 2) of Corollary \ref{coro_b}, we know $\mathrm{grad\mathcal{F}(\rho_\theta)|_{\mathcal{P}_{\Theta}}=\mathrm{grad}\mathcal{F}(\rho_\theta)}$ for all $\theta\in\Theta$. This indicates that there is no local error for our approximation, one can verify that the solution to the parametric Fokker-Planck equation also solves the original equation.

In addition to previous results, we have the following corollary:
\begin{corollary}
The solution of Fokker-Planck equation (\ref{FPE}) with condition(\ref{conditions_example}) is Gaussian distribution for all $t>0$.
\end{corollary}
\begin{proof}
If we denote $\{\Gamma_t,b_t\}$ as the solutions to (\ref{wass_grad_flow_1_example}),(\ref{wass_grad_flow_2_example}), set $\theta_t=(\Gamma_t,b_t)$, then $\rho_t={T_{\theta_t}}_{\#}p$ solves the Fokker Planck Equation (\ref{FPE}) with conditions (\ref{conditions_example}). Since the pushforward of Gaussian distribution $p$ by an affine transform $T_\theta$ is still a Gaussian, we conclude that for any $t>0$, the solution  $\rho_t={T_{\theta_t}}_{\#}p$ is always Gaussian distribution. This is already a well known result about Fokker-Planck equation. We reprove it under our framework.
\end{proof}
\section{Numerical examples for 1D Fokker-Planck equation}
Since the Wasserstein metric tensor $G$ has an explicit solution when dimension $d=1$, it is convenient to numerically compute ODE (\ref{wass_grad_flow_on_para_spc}). 

For example, we can choose a series of basis functions $\{\varphi_k\}_{k=1}^n$. Each $\varphi_k$ can be chosen as a sinusoidal function or a piece-wise linear function defined on a certain interval $[-l,l]$. It is also beneficial to choose orthogonal or near-orthogonal basis functions because they will keep the metric tensor $G$ far away from ill-posedness. We set $T_\theta(x)=\sum_{k=1}^{m}\theta_k\varphi_k(x)$\footnote{In application, carefully choosing $T_\theta$ which is not necessarily invertibile or smooth can still provide valid results.}. Then according to (\ref{Metric tensor 1 dimension}), we can compute $G$ as
\begin{equation}
G_{ij}(\theta)=\mathbb{E}_{\mathbf{X}\sim p} \Big[\varphi_i(\mathbf{X})\varphi_j(\mathbf{X})\Big] \quad 1\leq i,j\leq m \nonumber
\end{equation}
Recall that $F(\theta)=\int V(x)\rho_\theta(x) dx+\beta\int \rho_\theta(x)\log\rho_\theta(x)dx$. The second part of $F$ is the entropy of $\rho_\theta$, which can be computed by solving the following optimization problem \cite{EssidLaeferTabak2018_adaptivea}:
\begin{equation}
    \int \rho_\theta(x)\log\rho_\theta(x)~dx=\underset{h}{\mathrm{sup}}\Big\{\int  h(x)\rho_\theta(x)~dx-\int  e^{h(x)}dx\Big\}+1 \label{DV}
\end{equation}
We can solve (\ref{DV}) by parametrizing $h$. Suppose the optimal solution is $h^*$. Then by envelope theorem, we know $\nabla_\theta F(\theta)$ can be computed as
\begin{align}
\nabla_\theta F(\theta)&=\partial_\theta\left(\int V(x)\rho_\theta(x)~dx+\beta\int  h^*(x)\rho_\theta(x)~dx\right)\nonumber\\
&=\mathbb{E}_{\mathbf{x}\sim p}\Big[\partial_\theta T_\theta(\mathbf{X})^T\nabla_y(V(y)+\beta h^*(y))|_{y=T_\theta(\mathbf{X})})\Big] \label{gradF}
\end{align}
Notice that both the metric tensor $G$ and $\nabla_\theta F(\theta)$ are written in forms of expectations, thus we can compute them by Monte Carlo simulations. And finally, (\ref{wass_grad_flow_on_para_spc}) can be computed by forward Euler method. 

Our numerical results are always demonstrated by sample points: For each time node $t$, we sample points $\{\mathbf{X}_1,...,\mathbf{X}_N\}$ from $p$, then $\{T_{\theta_t}(\mathbf{X}_1),...,T_{\theta_t}(\mathbf{X}_N)\}$ are our numerical samples from distribution $\rho_t$ which solves the Fokker-Planck equation. 

Here are several numerical results based on our method. We exhibit them in the form of histograms. Consider the potential $V(x)=(x+1)^2(x-1)^2$. Suppose the initial distribution is $\rho_0=\mathcal{N}(0,I)$. Figure 1 contains histograms of $\rho_t$ which solves $\frac{\partial\rho}{\partial t}=\nabla\cdot(\rho\nabla V)$ at different time nodes; we know $\rho_t$ converges to $\frac{\delta_{-1}+\delta_{+1}}{2}$ as $t\rightarrow \infty$. Here $\delta_a$ is the Dirac distribution concentrated on point $a$. Figure 2 contains histograms of $\rho_t$ which solves $\frac{\partial\rho}{\partial t}=\nabla\cdot(\rho\nabla V)+\frac{1}{4}\Delta\rho$ at different time nodes, we know $\rho_t$ will converge to Gibbs distribution $\rho_*=\frac{1}{Z}\exp(-4(x+1)^2(x-1)^2)$, with $Z$ being a normalizing constant, as $t\rightarrow\infty$. The density function of $\rho_*$ is exhibited in Figure 2.
\begin{figure}[h]
\centering
\includegraphics[width=\textwidth]{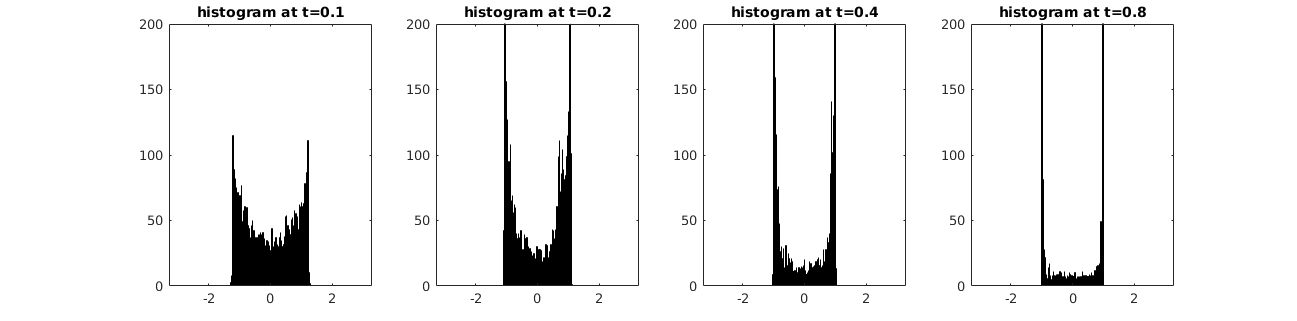}
\caption{Histograms of $\rho_t$ solving $\frac{\partial\rho}{\partial t}=\nabla\cdot(\rho\nabla V ) $}
\end{figure}
\begin{figure}[h]
\centering
\includegraphics[width=\textwidth]{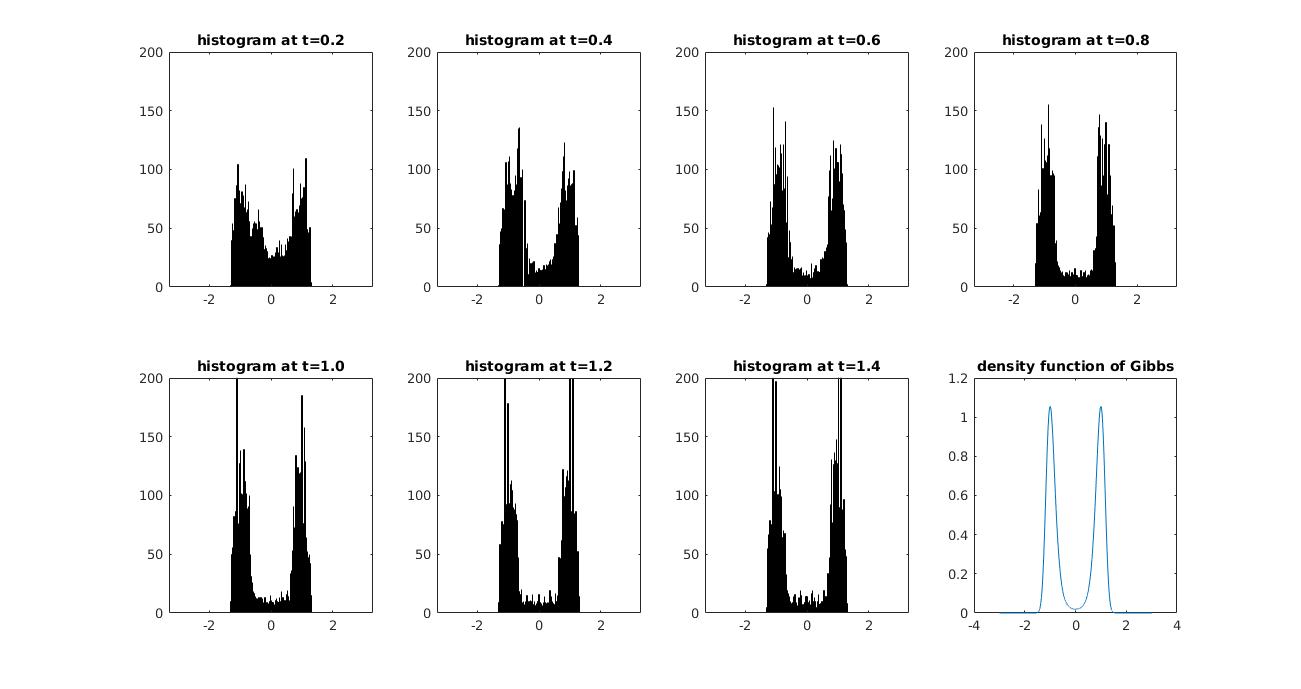}
\caption{Histograms of $\rho_t$ solving $\frac{\partial\rho}{\partial t}=\nabla\cdot(\rho\nabla V)+\frac{1}{4}\Delta\rho$}
\end{figure}

\newpage
\section{Discussion}
We presented a new approach for approximating Fokker-Planck equations by parameterized push-forward mapping functions. Compared to the classical moment method and MCMC method, we propose a systemic way for obtaining a finite dimensional ODE on parameter space. The ODE represents the evolution of statistical information conveyed in the original Fokker-Planck equation. In the future, we will study its geometric and statistical properties, and derive practical numerical methods for applications in scientific computing and machine learning. 

\noindent\textbf{Acknowledgement}
This project has received funding from AFOSR MURI FA9550-18-1-0502 and NSF Awards DMS–1419027, DMS-1620345, and ONR Award N000141310408. 


\bibliographystyle{abbrv}
\bibliography{references}

%
%
%





\end{document}